\documentclass{amsart}

\pdfoutput=1
\usepackage[margin=1.0in]{geometry}
\usepackage{amssymb}
\usepackage{amsmath,amsthm}
\usepackage{graphicx}
\usepackage{tikz,lipsum,lmodern}
\usetikzlibrary{calc}
\usepackage{caption} 
\usetikzlibrary{patterns}
\usetikzlibrary { decorations.pathmorphing, decorations.pathreplacing, decorations.shapes, }
\usepackage{graphicx}
\usepackage{ytableau}
\usepackage{float}
\usepackage{hyperref}
\usepackage{tikz-qtree}
\usetikzlibrary{shadows,trees}
\usepackage{colortbl}
\usepackage{color}
\usepackage{url}
\usepackage{amsfonts}
\usepackage{amsrefs}
\usepackage{multicol}
\usepackage{array}
\usepackage{multirow}
\usepackage{booktabs}
 \usepackage{cleveref}
\numberwithin{equation}{section}
\numberwithin{figure}{section}
\numberwithin{table}{section}

\newtheorem{theorem}{Theorem}[section]
\newtheorem{corollary}[theorem]{Corollary}

\newtheorem{conjecture}[theorem]{Conjecture}

\newtheorem{example}[theorem]{Example}
\newtheorem{problem}[theorem]{Problem}
\allowdisplaybreaks

\begin{document}

 \title{A New Symmetric Function Identity With an Application to symmetric group character values}

\author{Karlee J. Westrem}
\address{Department of Mathematical Sciences\\
Michigan Technological University\\ Houghton, MI 49931} \email{kjwestre@mtu.edu}
\date{September, 2024}
\begin{abstract}
Symmetric functions show up in several areas of mathematics including enumerative combinatorics and representation theory. 
Tewodros Amdeberhan conjectures equalities of $\Sigma_n$ characters sums over a new set called $Ev(\lambda)$. When investigating the alternating sum of characters for $Ev(\lambda)$ written in terms of the inner product of Schur functions and power sum symmetric functions, we found an equality between the alternating sum of power sum symmetric polynomials and a product of monomial symmetric polynomials. As a consequence, a special case of an alternating sum of $\Sigma_n$ characters over the set $Ev(\lambda)$ equals $0$.
\end{abstract}

\maketitle

\section{Introduction}
Inspired by conjectures of Tewodros Amdeberhan, we were led to consider this unusual multiset called $Ev(\lambda)$ and the alternating sum of power sum symmetric functions. This directed us toward a neat description of this as a product of monomial symmetric functions.
As an application with a new symmetric function identity, we provide some evidence toward some cases of the conjecture.
Recall that $\lambda=(\lambda_1, \lambda_2, \ldots, \lambda_r)$ is a partition of $n$, denoted $\lambda \vdash n$, if $\lambda_1 \geq \lambda_2 \cdots \geq \lambda_r>0$ are positive integers with $\sum_{i=1}^{r}\lambda_i=n$.
We denote the size of a partition $\lambda$ and its length by $|\lambda|$ and $\ell(\lambda)$, respectively. Let $\mathcal{P}$ be the set of all partitions and $p(n)$ be the number of partitions of size $n$. 
Amdeberhan defines $Ev(\lambda)$ to be the set of all partitions obtained by replacing each $\lambda_i$ with either $2\lambda_i$ (doubling) or $\lambda_i, \lambda_i$ (two copies) and rewritten in decreasing order. 
Note that the multiset $Ev(\lambda)$ has $2^{\ell(\lambda)}$ elements, each a partition of $2n$.

\begin{example}
    If $\lambda= (3,2,1)$, then $$Ev(\lambda) = \{ (6,4,2), (6,4,1^2), (6,2^3) (6,2^2,1^2), (4,3^2,2), (4,3^2,1^2), (3^2,2^3) (3^2,2^2,1^2) \}.$$ Note if $\lambda$ has repeated parts, for example $\lambda= (2,2,1,1)$, then $Ev(\lambda) = \{ (4^2,2^2), \overset{x2}{(4,2^4)} , \overset{x2}{(4,2^2,1^4)}, \overset{x2}{(4^2,2,1^2)}, $ $\overset{x2}{(2^5,1^2)}, (2^4,1^4), \overset{x4}{(4,2^3,1^2)}, (4^2,1^4), (2^6)  \}$. We see an appearance of products of binomial coefficients in the multiplicities.
\end{example}

Amdeberhan defines two other subsets of partitions of size $m$ 
by restricting either the rows or the columns to be of even length, denoted by, 
$$ \mathcal{R}_N(m)  := \{ \mu \vdash m : \ell(\mu) \leq N; \mu_i \text{ is even for all } i \}, $$ 
$$\mathcal{R}_N^c(m)  := \{ \mu \vdash m : \ell(\mu) \leq N; \mu_i' \text{ is even for all } i \} $$
respectively.
One of the conjectures sums together specific values in the character table of the symmetric group $\Sigma_n$. Recall that irreducible characters in the character table are labelled by partitions of $n$ and conjugacy classes are also labelled by partitions of $n$ via cycle type.
Define $\chi_{\lambda}^{\mu}$ as the irreducible character of $\mu$ on the conjugacy class $\lambda$ of the symmetric group. One can calculate any character of the symmetric group by using the Murnagan-Nakayama Rule \cite{Murnaghan1971TheCO}, \cite{Nakayama1940OnSM}.
For this paper, we compute the irreducible character values as certain inner products of symmetric polynomials.

 \section{Symmetric Functions}
 The character $\chi_{\lambda}^{\mu}$ can be written as an inner product of power sum symmetric functions $p_{\lambda}$ and Schur functions $s_{\mu}$, which are two different vector space bases of $\Lambda^n$, the set of all homogeneous symmetric functions of degree $n$, with $\operatorname{dim } \Lambda^n = p(n)$. A third basis is the set of monomial symmetric functions.
 
If $\lambda = (\lambda_1, \lambda_2, \dots) \vdash n$, define a monomial symmetric function $m_{\lambda}(x_1,x_2,x_3,\dots) \in \Lambda^n$, with
$$ m_{\lambda} = \sum_{\alpha}{(x_1,x_2,x_3,\dots)}^{\alpha},$$
where the sum ranges over all \textit{distinct} permutations $\alpha = (\alpha_1, \alpha_2, \dots)$ of the entries of the vector $\lambda = (\lambda_1, \lambda_2, \dots )$.
For example, $ m_{11} = \sum_{i<j} x_i x_j $.
Then the set of \textit{power sum symmetric functions} that are indexed by $\lambda \in \mathcal{P}$ and are defined as follows,
$$p_n =m_n = \sum_{i} x_i^n, \:\: n \geq 1 \:\:\:\:\:\: (\text{with } p_0 = m_{\emptyset} = 1) $$
 $$p_{\lambda} = p_{\lambda_1}p_{\lambda_2}\cdots \:\: \text{ if } \lambda = (\lambda_1, \lambda_2, \dots).$$
Young's Rule \cite{sagan}, expresses Schur polynomials $s_{\lambda}$ as linear combinations of monomial symmetric functions, where the coefficients are Kostka numbers given by 
$$ s_{\lambda} = \sum_{\mu} K_{\lambda \mu} m_{\mu}. $$
Now \cite[Corollary 17.5]{StanleyEC2} states, $$ \chi^{\mu}_{\lambda} = \langle p_{\lambda},s_{\mu} \rangle.$$

We state the main result of the paper, which will later be applied to \cite[Corollary 17.5]{StanleyEC2} to simplify Amdeberhan's conjecture.

\begin{theorem}\label{main-result}
  If $\lambda = (\lambda_1,\lambda_2,
\dots, \lambda_r) \vdash n$, then
$$ \sum_{\Tilde{\lambda} \in Ev(\lambda)} (-1)^{\ell{(\Tilde{\lambda})}} p_{\Tilde{\lambda}} =  2^r \prod_{1 \leq i \leq r} m_{\lambda_i \lambda_i}.$$  
\end{theorem}

\begin{example}
    Suppose $\lambda = (2,1)$. Then $$\sum_{\tilde{\lambda}} (-1)^{\ell(\tilde{\lambda})} p_{\tilde{\lambda}} = p_{42}-p_{411}-p_{222}+p_{2211}$$
    $$ = m_4m_2-m_4m_1m_1-m_2m_2m_2+m_2m_2m_1m_1$$
    $$ = m_{42}+m_6 - 2m_{411} -m_{42}-2m_{51} + m_6 -6m_{222}-3m_{42}-m_6$$
    $$ +4m_{2211}+6m_{222} +4m_{321}+4m_{33} + 2m_{411} +3m_{42}+2m_{51}+m_6$$
    $$ = 4m_{2211}+4m_{321}+4m_{33}$$
    $$ = 2^2 \cdot m_{22}m_{11}.$$
\end{example}
\begin{proof}
Let $\lambda=(\lambda_1, \lambda_2, \ldots, \lambda_r) \vdash n$.
Consider $$\sum_{\tilde{\lambda} \in
Ev(\lambda)}(-1)^{l(\tilde{\lambda})}p_{\tilde{\lambda}}.$$ 
This is an
alternating sum of $2^r$ terms of the form:

\begin{equation}
\label{eq: expansion}
   \left(
     \begin{array}{c}
       p_{2\lambda_1} \\
       \text{or} \\
       p_{\lambda_1}p_{\lambda_1}
     \end{array}
   \right)  \left(
     \begin{array}{c}
       p_{2\lambda_2} \\
       \text{or} \\
       p_{\lambda_2}p_{\lambda_2}
     \end{array}
   \right) \cdots \left(
     \begin{array}{c}
       p_{2\lambda_r} \\
       \text{or} \\
       p_{\lambda_r}p_{\lambda_r}
     \end{array}
   \right).
\end{equation}

Our goal is to track the large number of cancellations occurring in the
alternating sum. Rather than simply computing the coefficient of each
monomial it will be helpful to keep track of not just the monomial but
the choices of terms in the expansion \eqref{eq: expansion} that
produced it. Thus we write monomials appearing in \eqref{eq: expansion} as:
\begin{equation}\label{eq: monomial}
   x_{i_1}^{\lambda_1} x_{j_1}^{\lambda_1}  x_{i_2}^{\lambda_2}
x_{j_2}^{\lambda_2} \cdots   x_{i_r}^{\lambda_r} x_{j_r}^{\lambda_r}
\end{equation}
where a choice of $p_{2\lambda_k}$ in \eqref{eq: expansion} forces
$i_k=j_k$ in \eqref{eq: monomial} whereas a choice of
$p_{\lambda_k}p_{\lambda_k}$ allows either $i_k=j_k$ or $i_k \neq j_k.$
We do not attempt to simplify or rearrange the order of the variables,
so each monomial will have $2r$ terms.

Now we can compute the coefficients of each monomial in the expansion.
Given a monomial as in \eqref{eq: monomial}, let t denote the number of
positions with $i_s=j_s$, so $0\leq t \leq r$. For positions with
$i_s=j_s$ the monomial can arise from either the choice of
$p_{\lambda_s}p_{\lambda_s}$ or $p_{2\lambda_s}$. We can choose any
subset of the $t$ positions to select $p_{\lambda_s}p_{\lambda_s}$, with
the complementary positions having $p_{2\lambda_s}$. A subset of size
$m$ corresponds to $\tilde{\lambda}$ of length $2r-m$. For the remaining
$r-t$ positions with $i_s \neq j_s$ we get a factor of two representing
the possible choice of $x_{i_s}^{\lambda_s}x_{j_s}^{\lambda_s}$ or
$x_{j_s}^{\lambda_s}x_{i_s}^{\lambda_s}$ from the term
$p_{\lambda_s}p_{\lambda_s}.$

Thus for $t>0$ we get a total coefficient of:

\begin{equation*}\label{eq: sumcancels}
   2^{r-t}\left[\binom{t}{0}-\binom{t}{1}+\binom{t}{2}-\cdots \pm
\binom{t}{t}\right]=0.
\end{equation*}

For $t=0$ the formula reduces to $(-1)^{2r}2^r=2^r$. Thus:

\begin{equation}\label{eq: answer}
\sum_{\tilde{\lambda} \in
Ev(\lambda)}(-1)^{l(\tilde{\lambda})}p_{\tilde{\lambda}}=2^r\sum_{i_k
\neq j_k \forall
k}x_{i_1}^{\lambda_1}x_{j_1}^{\lambda_1}x_{i_2}^{\lambda_2}x_{j_2}^{\lambda_2}\cdots
x_{i_r}^{\lambda_r}x_{j_r}^{\lambda_r}=2^r\prod_{i=1}^{r}m_{\lambda_i
\lambda_i}.
\end{equation}
\end{proof}

We find that this result allows us to prove some facts about one of Amdeberhan's conjectures. The conjecture below stems from a wider conjecture on q-series provided in \cite{460090}.

\section{Application to Character Table Conjecture}


\begin{conjecture}\cite{460090}\label{conjecture-characters}
   If $\lambda \vdash n$, then $$ \sum_{\tilde{\lambda} \in Ev(\lambda), \mu \in \mathcal{R}_{2N+1}(2|\lambda|)}(-1)^{\ell(\tilde{\lambda})}\chi^{\mu}_{\tilde{\lambda}} = \sum_{\tilde{\lambda} \in Ev(\lambda), \mu \in \mathcal{R}^c_{2N}(2|\lambda|)}\chi^{\mu}_{\tilde{\lambda}}.$$
   \end{conjecture}
If $N$ is large enough, then the size of the parts or the number of parts of $\mu$ is not restricted. So we can pair $\mu \in R_{\infty} (2|\lambda|)$ with its conjugate $\mu' \in R^c_{\infty} (2|\lambda|)$ since
    $$ \chi^{\mu}_{\tilde{\lambda}} = (-1)^{|\tilde{\lambda}| + \ell(\tilde{\lambda})} \cdot \chi^{\mu'}_{\tilde{\lambda}}$$ for any $\tilde{\lambda}$.   


Let us consider a few examples such as when $\lambda =(2^2,1)$ and $\lambda= (2^4,1^2)$ to demonstrate the equality and show the cancellations happening. For $\lambda =(2^2,1)$, we provide a partial character table of $\Sigma_{10}$, the table with the signs and multiplicity in $Ev(\lambda)$, and a table with just multiplicity.

\begin{example}
If $\lambda = (2^2,1)$, then $Ev((2,2,1)) = \{(4^2,2),(4^2,1^2), \overset{x2}{(4,2^3)},\overset{x2}{(4,2^2,1^2)}, (2^5), (2^4,1^2)\}$. In \ref{table:1}, we restrict the character table down to partitions appearing in $Ev(\lambda)$ for the columns and partitions in both $R_{2N+1}(2|\lambda|)$ and $R^c_{2N}(2|\lambda|)$ appearing in the rows. Now in Table \ref{table:2} and Table \ref{table:3}, we indicate in red the multiplicity for partitions in $Ev(\lambda)$ along $(-1)^{\ell(\tilde{\lambda})}$ for \ref{table:2}. Table \ref{table:2} specifically pertains to partitions in $R_{2N+1}(2|\lambda|)$ and Table \ref{table:3} for partitions in $R_{2N}^c(2|\lambda|)$. We keep track of the row sum and note which $N = 1,2,3,4$ or $5$ the row sum counts for relating to the conjecture. \\

\begin{center}
    \setlength{\tabcolsep}{3pt}
\begin{table}[H]
\begin{tabular}{ c c c c c c c c c c c c c } 
 & $[2^4 1^2]$	& $[2^5]$   & 	$[4^1 2^2 1^2]$ &	$[4^1 2^3]$  	& $[4^2 1^2]$ &	$[4^2 2^1]$ 	& & & &  &  \\
 \cline{1-8}
$s[10^1]$	&	1 & 1		&	 1 &  1	&	1	&  1 &  &   &	&  &   &  \\
$s[8^1 2^1]$	&3	& 5	& 1	& 3	&	-1&	 1 &  &  \\
$s[6^1 4^1]$		&2	& 10		& 0	& 4		&	2	& 2 &  &  \\
$s[4^2 2^1]$	&	4	& 20&	2&	2	&	0&	0 &  &  \\
$s[6^1 2^2]$	&	9	& 15	&	1& 3&	1&	-1 & &  \\
$s[4^1 2^3]$ &	4	&20&	-2	&-2	&	0&	0 & & \\
$s[2^5]$	&2&	10&		0 &	-4&		2&	2 &  &  \\
$s[5^2]$ &		$2$ &	 $-10$ &		 $0$ &	-4	&	$2$ &	 $-2$	&   &  \\
$s[4^2 1^2]$	&	4& -20&	2&	-2&		0&	0 &   &  \\
$s[3^2 2^2]$ &	4&	-20&	-2&	2	&	0&	0 &  &  \\
$s[3^2 1^4]$&		9&	-15	&	-1&3	&	1&	1 &  &  \\
$s[2^4 1^2]$&		2	&-10&	0&	4&	2&	-2	& & \\
$s[2^2 1^6]$		&3	&-5&	-1	&3	&-1&	-1 &  & \\
$s[2^1 1^8]$		&1&1	&	-1&	-1	&	1&	1 &  & \\
$s[1^{10}] $ &	1 &	-1	& -1 &	1	& 1 &	-1 &   \\
  \end{tabular}
  \caption{Partial Character Table for $\Sigma_{10}$}
   \label{table:1}
\end{table}
\end{center}

\begin{center}
    \setlength{\tabcolsep}{3pt}
\begin{table}[H]
\begin{tabular}{ c c c c c c c c c c c c c } 
 & $[2^4 1^2]$	& $[2^5]$   & 	$[4^1 2^2 1^2]$ &	$[4^1 2^3]$  	& $[4^2 1^2]$ &	$[4^2 2^1]$ 	& Row Sum & & &  &  \\
 \cline{1-8}
$s[10^1]$	&	1 &	\textcolor{red}{-1}$\cdot$1		&	\textcolor{red}{-2}$\cdot$1 & \textcolor{red}{2}$\cdot$1	&	1	& \textcolor{red}{-1}$\cdot$1 & $ 0$  &  \multirow{1}{*}{\(\left. \begin{array}{c} \\ \\ \\ \\ \\ \end{array} \right\}\text{\rotatebox{90}{N=1}}\)} & \multirow{1}{*}{\(\left. \begin{array}{c}  \\  \\ \\ \\ \\ \\ \\ \end{array} \right\}\text{\rotatebox{90}{N=2,3,4,5}}\)}	&  &   &  \\

$s[8^1 2^1]$	&3	&\textcolor{red}{-1}$\cdot$5	&\textcolor{red}{-2}$\cdot$1	&\textcolor{red}{2}$\cdot$3	&	-1&	\textcolor{red}{-1}$\cdot$1 & $ 0$ &  \\
$s[6^1 4^1]$		&2	&\textcolor{red}{-1}$\cdot$10		&\textcolor{red}{-2}$\cdot$0	&\textcolor{red}{2}$\cdot$4		&	2	&\textcolor{red}{-1}$\cdot$2 & $ 0$ &  \\
$s[4^2 2^1]$	&	4	&\textcolor{red}{-1}$\cdot$20&	\textcolor{red}{-2}$\cdot$2&	\textcolor{red}{2}$\cdot$2	&	0&	\textcolor{red}{-1}$\cdot$0 & $-16$ &  \\
$s[6^1 2^2]$	&	9	&\textcolor{red}{-1}$\cdot$15	&	\textcolor{red}{-2}$\cdot$1&	\textcolor{red}{2}$\cdot$3&	1&	\textcolor{red}{-1}$\cdot$-1 & $ 0$ &  \\
$s[4^1 2^3]$ &	4	&\textcolor{red}{-1}$\cdot$20&	\textcolor{red}{-2}$\cdot$-2	&\textcolor{red}{2}$\cdot$-2	&	0&	\textcolor{red}{-1}$\cdot$0 & $-16$ & \\
$s[2^5]$	&2&	\textcolor{red}{-1}$\cdot$10&		\textcolor{red}{-2}$\cdot$0&	\textcolor{red}{2}$\cdot$-4&		2&	\textcolor{red}{-1}$\cdot$2 & $ -16$  &  \\
  \end{tabular}
  \caption{Partial Character Table for $\Sigma_{10}$ with binomial coefficients and signs}
   \label{table:2}
\end{table}
\end{center}

\begin{center}
    \setlength{\tabcolsep}{3pt}
\begin{table}[H]
\begin{tabular}{ c c c c c c c c c c c c c } 
 & $[2^4 1^2]$	& $[2^5]$   & 	$[4^1 2^2 1^2]$ &	$[4^1 2^3]$  	& $[4^2 1^2]$ &	$[4^2 2^1]$ 	& Row Sum & & &  &  \\
 \cline{1-8}
 & & & & & & & &  \multirow{1}{*}{\(\left. \begin{array}{c} \\ \end{array} \right\}\text{\rotatebox{90}{N=1}}\)} & \multirow{1}{*}{\(\left. \begin{array}{c}  \\ \\ \\ \\ \end{array} \right\}\text{\rotatebox{90}{N=2}}\)}	& \multirow{6}{*}{\(\left. \begin{array}{c}  \\ \\ \\ \\ \\ \\ \\ \end{array} \right\}\text{\rotatebox{90}{N=3}}\)}  &  \multirow{6}{*}{\(\left. \begin{array}{c}  \\ \\ \\ \\  \\ \\ \\ \\ \\ \end{array} \right\}\text{\rotatebox{90}{N=4,5}}\)} \\
$s[5^2]$ &		$2$ &	\textcolor{red}{1}$\cdot$ $-10$ &		\textcolor{red}{2}$\cdot$ $0$ &	\textcolor{red}{2}$\cdot$-4	&	$2$ &	\textcolor{red}{1}$\cdot$ $-2$	& $-16$  &  \\
$s[4^2 1^2]$	&	4&	\textcolor{red}{1}$\cdot$-20&	\textcolor{red}{2}$\cdot$2&	\textcolor{red}{2}$\cdot$-2&		0&	\textcolor{red}{1}$\cdot$0 & $ -16$  &  \\
$s[3^2 2^2]$ &	4&	\textcolor{red}{1}$\cdot$-20&	\textcolor{red}{2}$\cdot$-2&	\textcolor{red}{2}$\cdot$2	&	0&	\textcolor{red}{1}$\cdot$0 & $-16$  &  \\
$s[3^2 1^4]$&		9&	\textcolor{red}{1}$\cdot$-15	&	\textcolor{red}{2}$\cdot$-1&	\textcolor{red}{2}$\cdot$3	&	1&	\textcolor{red}{1}$\cdot$1 & $0$  &  \\
$s[2^4 1^2]$&		2	&\textcolor{red}{1}$\cdot$-10&	\textcolor{red}{2}$\cdot$0&	\textcolor{red}{2}$\cdot$4&	2&	\textcolor{red}{1}$\cdot$-2	& $0$ & \\
$s[2^2 1^6]$		&3	&\textcolor{red}{1}$\cdot$-5&		\textcolor{red}{2}$\cdot$-1	&\textcolor{red}{2}$\cdot$3	&-1&	\textcolor{red}{1}$\cdot$-1 & $ 0$ & \\
$s[2^1 1^8]$		&1&	\textcolor{red}{1}$\cdot$1	&	\textcolor{red}{2}$\cdot$-1&	\textcolor{red}{2}$\cdot$-1	&	1&	\textcolor{red}{1}$\cdot$1 & $ 0$ & \\
$s[1^{10}] $ &	1 &	\textcolor{red}{1}$\cdot$-1	& \textcolor{red}{2}$\cdot$-1 &	\textcolor{red}{2}$\cdot$1	& 1 &	\textcolor{red}{1}$\cdot$-1 & $0$&  \\
  \end{tabular}
  \caption{Partial Character Table for $\Sigma_{10}$ with binomial coefficients}
  \label{table:3}
\end{table}
\end{center}

We see for $N=1$,
$$ \sum_{\tilde{\lambda} \in Ev(\lambda), \mu \in \mathcal{R}_{3}(2|\lambda|)}(-1)^{\ell(\tilde{\lambda})}\chi^{\mu}_{\tilde{\lambda}} = \sum_{\tilde{\lambda} \in Ev(\lambda), \mu \in \mathcal{R}^c_{2}(2|\lambda|)}\chi^{\mu}_{\tilde{\lambda}} = -16.$$
For $N=2,3,4,5$, we get every partition so,
$$ \sum_{\tilde{\lambda} \in Ev(\lambda), \mu \in \mathcal{R}_{2N+1}(2|\lambda|)}(-1)^{\ell(\tilde{\lambda})}\chi^{\mu}_{\tilde{\lambda}} = \sum_{\tilde{\lambda} \in Ev(\lambda), \mu \in \mathcal{R}^c_{2N}(2|\lambda|)}\chi^{\mu}_{\tilde{\lambda}} = -48.$$


\end{example}

\begin{example}

If $\lambda = (2^4,1^2)$, note the character values become larger in the character table and we have larger binomial coefficients appearing in $Ev((2^4,1^2)) = \{ (4^3,2^2), \overset{x2}{(4^3,2,1^2)},	(4^3,1^4) ,	\overset{x3}{(4^2,2^4)}, \overset{x6}{(4^2,2^3,1^2)}, $ \\ $\overset{x3}{(4^2,2^2,1^4)},	\overset{x3}{(4,2^6)},	 \overset{x6}{(4,2^5,1^2)}, \overset{x3}{(4,2^4,1^4)}, (2^8),	  \overset{x2}{(2^7,1^2)},	(2^6,1^4), \overset{x2}{(4,2^2,1^4)} \}$.

If $N=1$, then
$$ \sum_{\tilde{\lambda} \in Ev(\lambda), \mu \in \mathcal{R}_{3}(2|\lambda|)}(-1)^{\ell(\tilde{\lambda})}\chi^{\mu}_{\tilde{\lambda}} = \sum_{\tilde{\lambda} \in Ev(\lambda), \mu \in \mathcal{R}^c_{2}(2|\lambda|)}\chi^{\mu}_{\tilde{\lambda}} = 224.$$
When $N=2$, we have 
$$ \sum_{\tilde{\lambda} \in Ev(\lambda), \mu \in \mathcal{R}_{2N+1}(2|\lambda|)}(-1)^{\ell(\tilde{\lambda})}\chi^{\mu}_{\tilde{\lambda}} = \sum_{\tilde{\lambda} \in Ev(\lambda), \mu \in \mathcal{R}^c_{2N}(2|\lambda|)}\chi^{\mu}_{\tilde{\lambda}} = 2176.$$

If $N=3$, then
$$ \sum_{\tilde{\lambda} \in Ev(\lambda), \mu \in \mathcal{R}_{2N+1}(2|\lambda|)}(-1)^{\ell(\tilde{\lambda})}\chi^{\mu}_{\tilde{\lambda}} = \sum_{\tilde{\lambda} \in Ev(\lambda), \mu \in \mathcal{R}^c_{2N}(2|\lambda|)}\chi^{\mu}_{\tilde{\lambda}} = 3616$$
and 
$N \geq 4$ which gives the remainder of the set as,
$$ \sum_{\tilde{\lambda} \in Ev(\lambda), \mu \in \mathcal{R}_{2N+1}(2|\lambda|)}(-1)^{\ell(\tilde{\lambda})}\chi^{\mu}_{\tilde{\lambda}} = \sum_{\tilde{\lambda} \in Ev(\lambda), \mu \in \mathcal{R}^c_{2N}(2|\lambda|)}\chi^{\mu}_{\tilde{\lambda}} = 3840.$$

\end{example}

For an arbitrary $\lambda$  it's not quite clear how the cancellations happens, as seen in Table \ref{table:2} and Table \ref{table:3}. But we can verify that if $\mu_1 >n$, the sum row in the partial character table will be zero. Using Theorem \ref{main-result}, we rewrite the left hand side of Conjecture \ref{conjecture-characters} as, 
$$ \sum_{\tilde{\lambda} \in Ev(\lambda), \mu \in R_{2N+1}(2|\lambda|)} (-1)^{\ell(\tilde{\lambda})} \chi_{\tilde{\lambda}}^{\mu} = \langle \sum_{\tilde{\lambda}}(-1)^{\ell(\tilde{\lambda})} p_{\tilde{\lambda}}, \sum_{\mu \in R_{2N+1}(2|\lambda|)} s_\mu \rangle = \langle 2^r \prod_{i=1}^{r} m_{\lambda_i, \lambda_i}, \sum_{\mu \in R_{2N+1}(2|\lambda|)} s_\mu \rangle.$$  

We now state the corollary for the sum row being zero when $\mu_1 >n$ following from Theorem \ref{main-result}.

\begin{corollary} \label{corollary: 1}
    If $\mu_1 >n$, then $\sum_{\tilde{\lambda} \in Ev(\lambda), \mu \in R_{N}(2|\lambda|)} (-1)^{\ell(\tilde{\lambda})}\chi_{\tilde{\lambda}}^{\mu} = 0$.
\end{corollary}

\begin{proof}
    We first note the inverse Kostka number $K^{-1}(\tau, \mu) = 0$ unless $\tau \unrhd \mu.$ Young's rule tells us that $m_{\tau} = \sum K^{-1} (\tau,\mu) s_{\mu}$.
Since Schur polynomials form an orthonormal basis, we have that
$\langle s_{\sigma}, s_{\mu} \rangle = \delta_{\sigma \mu}.$

Therefore the inner product
 $$ \Big\langle \prod_{i=1}^{r} m_{\lambda_i \lambda_i}, \sum_{\mu} s_{\mu} \Big\rangle  = \Big\langle \sum_{\tau_1 \leq n} m_{\tau}, \sum_{\mu} s_{\mu}
    \Big\rangle,$$
where we have to only consider the tableau $\mu$ when $\mu_1 \leq n$, since the product of monomial functions has the restriction of $\tau_1 \leq n$.

Therefore,
$$\sum_{\tilde{\lambda} \in Ev(\lambda), \mu \in R_{N}(2|\lambda|)} (-1)^{\ell(\tilde{\lambda})}\chi_{\tilde{\lambda}}^{\mu} = 0$$ when $\mu_1 > n.$
\end{proof}

So the conjecture is simplified to,
$$ \sum_{\tilde{\lambda} \in Ev(\lambda), \mu \in \mathcal{R}_{2N+1}(2|\lambda|), \mu_1 \leq n  }(-1)^{\ell(\tilde{\lambda})}\chi^{\mu}_{\tilde{\lambda}} = \sum_{\tilde{\lambda} \in Ev(\lambda), \mu \in \mathcal{R}^c_{2N}(2|\lambda|)} \chi^{\mu}_{\tilde{\lambda}}.$$ 

\section{Future Direction}
The next step of this project is to prove Amdeberhan's conjecture holds for small values of $N$. We want to describe how the cancelling is happening in both summations. 

\begin{problem}
    Are the other identities that allow us to simplify the calculations or will we need to use the Murnaghan-Nakayama rule to calculate the character values to show equality?
\end{problem}


\section{Acknowledgements}
I would like thank Tewodros Amdeberhan for sharing his ideas with me.
I would also like to thank my advisor David Hemmer for providing guidance on this problem and words of wisdom throughout the process.

\bibliography{SymmetricFunctionsReferences}

\end{document}